\documentclass[11pt,leqno]{article}
\usepackage{amsmath,amssymb,mathrsfs,amsthm,bm,ddag}

\usepackage[normalem]{ulem}
\usepackage{xy}
\input{xy}
\xyoption{all}

\usepackage{threeparttable}

\usepackage{xspace}
\usepackage[geometry]{ifsym}
\makeatletter
\def\@seccntformat#1{\csname the#1\endcsname.\hspace{2ex}}

 \renewcommand{\subsection}%
  {\@startsection{subsection}%
  {2}%
  {\z@}%
  {2ex}%{-3.5ex plus -1ex minus -.2ex}
  {0ex}%{-1.5ex}%
  {\reset@font\normalsize\bfseries}}%

 \@addtoreset{equation}{subsection}
 \newcommand{\nsection}{\@startsection{section}{1}{\z@}%
     {-5ex}%-3.5ex plus-1ex minus-.2ex}%
     {1ex}%2.3ex plus.2ex}%
     {\reset@font\center\large\sc}}

 \renewenvironment{thebibliography}[1]
 {\nsection*{\refname\@mkboth{\refname}{\refname}}%
   \list{\@biblabel{\@arabic\c@enumiv}}%
   {\settowidth
   \labelwidth{\@biblabel{#1}}%
   \leftmargin
	\labelwidth
        \advance
	 \leftmargin
	 \labelsep
         \@openbib@code
         \usecounter{enumiv}%
         \let\p@enumiv\@empty
	 \parskip=0pt
	 \itemsep=1pt
	 \parsep=1pt
	 \itemindent=\z@
         \renewcommand\theenumiv{\@arabic\c@enumiv}}%
   	 \sloppy
   	 \clubpenalty4000
   	 \@clubpenalty\clubpenalty
   	 \widowpenalty4000%
   	 \footnotesize
   	 \sfcode`\.\@m}
  	 {\def\@noitemerr
    	 {\@latex@warning{Empty `thebibliography' environment}}%
   	 \endlist}

\makeatother

\newtheoremstyle{thm}
 {1em}% Space above
 {3pt}% Space below
 {\itshape}% Body font
 {}% Indent amount
 {\bf}% Theorem head font
 {. ---}% Punctuation after theorem head
 {0.5em}% Space after theorem head
 {}% Theorem head spec (can be left empty, meaning `normal')
\newtheoremstyle{dfn}
 {1em}% Space above
 {3pt}% Space below
 {}% Body font
 {}% Indent amount
 {\bf}% Theorem head font
 {. {---}}% Punctuation after theorem head
 {0.5em}% Space after theorem head
 {}% Theorem head spec (can be left empty, meaning `normal')

\swapnumbers
\theoremstyle{thm}
\newtheorem{thm}[subsection]{Theorem}
\newtheorem{lem}[subsection]{Lemma}
\newtheorem*{lem*}{Lemma}
\newtheorem{cor}[subsection]{Corollary}
\newtheorem*{cor*}{Corollary}
\newtheorem{prop}[subsection]{Proposition}
\newtheorem*{prop*}{Proposition}

\newtheorem*{conj*}{Conjecture}
\newtheorem*{thm*}{Theorem}

\theoremstyle{dfn}
\newtheorem{dfn}[subsection]{Definition}
\newtheorem*{dfn*}{Definition}

\newtheorem*{ex*}{Example}
\newtheorem{rem}[subsection]{Remark}
\newtheorem*{rem*}{Remark}

\newtheorem{prob}[subsection]{Question}

\newcommand{\shom}{\mathop{\mc{H}om}\nolimits}

\usepackage{color}
\newenvironment{meta1}{
\noindent\color{red}
\sffamily[}{\upshape]}
\newenvironment{meta2}{
\noindent\color{magenta}
\sffamily[}{\upshape]}

\newsavebox{\circlebox}
\savebox{\circlebox}{\fontencoding{OMS}\selectfont\char13}
\newlength{\circleboxwdht}

\renewcommand{\H}{\ms{H}}

\newcommand{\cH}{{}^\mr{c}\!\ms{H}}

\newcommand{\dd}[1]{[\![#1]\!]}

\newcommand{\slog}{\mc{L}og}

\begin{document}
\title{Around the nearby cycle functor for arithmetic $\ms{D}$-modules}
\author{Tomoyuki Abe}
\date{}
\maketitle

\begin{flushright}
 {\it Dedicated to Professor Shuji Saito\\
 on the occasion of his 60th birthday}
\end{flushright}

\begin{abstract}
 We will establish a nearby and vanishing cycle formalism for the
 arithmetic $\ms{D}$-module theory following Beilinson's philosophy.
 As an application, we define smooth objects in the framework of
 arithmetic $\ms{D}$-modules whose category is equivalent to the
 category of overconvergent isocrystals.
\end{abstract}

\section*{Introduction}
In this paper, we establish a theory of nearby/vanishing cycle functor
in the framework of arithmetic $\ms{D}$-modules and give some
applications. Unipotent nearby/vanishing cycle formalism has already
been established by the author together with D. Caro in \cite{AC2}
after the philosophy of Beilinson. Beilinson's philosophy
(cf.\ \cite[Remark after Corollary 3.2]{Beglue}) also tells us
how to go from unipotent nearby/vanishing cycle functors to the full
ones, and in fact, this philosophy underlies the argument
of \cite[Lemma 2.4.13]{A}.
The aim of this article is to carry this out more systematically so that
the nearby/vanishing cycle formalism is also accessible in the $p$-adic
cohomology theory.

Now, let us clarify what properties make full nearby/vanishing cycle
functors different from unipotent counterpart. Let $k$ be a perfect
field of positive characteristic. Given a morphism of finite type
$f\colon X\rightarrow\mb{A}^1_k$, and a ``$p$-adic coefficient object''
$\ms{M}$, we have already defined unipotent nearby/vanishing cycles
$\Psi^{\mr{un}}_f(\ms{M})$ and $\Phi^{\mr{un}}_f(\ms{M})$ as objects on
$X_0:=X\times_{\mb{A}^1}\{0\}$.
These functors are compatible with pushforward by proper morphisms and
pullback by smooth morphisms.
An important property of full nearby cycle functor is that it computes
the ``cohomology of the generic fiber'' when $f$ is proper.
However, $\Psi^{\mr{un}}$ is not powerful enough to compute the
cohomology. Let us explain what this means. Consider the simplest
possible situation, namely $X=\mb{A}^1$ and $f=\mr{id}$.
Consider the $p$-adic coefficient $\ms{M}$ defined by the differential
equation
\begin{equation*}
 x^2\partial-\pi=0.
\end{equation*}
This differential equation has singularity at $0$, and in fact, we may
prove that the equation is trivialized by an Artin-Schreier type
covering. The ``cohomology'' of $\ms{M}$ around the generic point of
$\mb{A}^1$ is merely the ``fiber'' of $\ms{M}$ at the generic point
because we are taking $f=\mr{id}$.
Thus, this should be some vector
space of dimension equal to the rank of $\ms{M}$, which is $1$ in this
situation. In particular,
$\Psi_{\mr{id}}(\ms{M})$ should not be zero. However, we may
compute that $\Psi^{\mr{un}}_{\mr{id}}(\ms{M})=0$. Thus,
$\Psi^{\mr{un}}_{\mr{id}}(\ms{M})$ does not meet our need.
Beilinson suggests to consider
$\bigoplus_{\ms{L}}\Psi^{\mr{un}}(\ms{M}\otimes\ms{L})$ where $\ms{L}$
runs over all the irreducible ``local system on a disk''.
In the situation above, $\ms{M}^{\vee}$ (where $(-)^\vee$ denotes the
dual) should be considered as an irreducible local system on the disk
around $0\in\mb{A}^1$. Thus the contribution from
$\Psi^{\mr{un}}(\ms{M}\otimes\ms{M}^\vee)$ does not vanish, which gives
us the correct computation of the cohomology of the generic fiber in
terms of nearby cycle functor.

Even though it is straightforward what to do philosophically, some
technical issues come in. First of all, the unipotent nearby/vanishing
cycle functors we have already defined {\em a priori} depends on the
choice of ``parameter'', whereas it should not be ideally.
This issue is treated in \S\ref{dfnfunc}.
Secondly, it is not clear from the definition that $\Psi_f$ and $\Phi_f$
have certain finiteness property.
We argue as Deligne to show the finiteness in \S\ref{finthm}.
After constructing nearby/vanishing cycle functors, we give small
applications. In \S\ref{smsec}, we define the category of smooth objects
intrinsically, and show that this category coincides with the category
of overconvergent isocrystals. We also show that this category is stable
under taking pushforward by proper and smooth morphism.
In the final section, \S\ref{locsec}, we propose a category over a
henselian trait which is an analogue of that of $\ell$-adic sheaves, and
show that our nearby/vanishing cycle functors factor through this
category.

\subsection*{Acknowledgment}\mbox{}\\
It is my great pleasure to dedicate this article to Professor Shuji
Saito, with deep respect to him and his mathematics, on the occasion of
his 60th birthday.
As a supervisor, Professor Saito taught me what it is to study
mathematics, encouraged me strongly in many occasions, advised me both
on mathematics and on life.
Without him, my life would not have been as rich.

This research is supported by Grant-in-Aid for Young Scientists (A)
16H05993.

\section{Vanishing cycle functor}
\label{dfnfunc}
\subsection{}
In the whole paper, we fix a (geometric) base tuple $(k,R,K,L)$ (cf.\
\cite[1.4.10, 2.4.14]{A}).
This is a collection of data where $k$ is a perfect
field of characteristic $p>0$, $R$ is a discrete valuation ring whose
residue field is $k$ such that some power of Frobenius automorphism on
$k$ lifts to $R$, $K:=\mr{Frac}(R)$, and $L$ is an algebraic extension
of $K$. Once we fix these data, we are able to define the $L$-linear
triangulated category $D(X)$ for a separated scheme $X$ of finite type
over $k$. This triangulated category is denoted by
$D(X/L_{\emptyset})$ or $D(X/\mf{T})$ where $\mf{T}$ is the fixed base
tuple to be more precise in \cite{A}.
When $L=K$ and $X$ is quasi-projective (or more generally,
realizable), we have a classical and more familiar description of $D(X)$
in terms of arithmetic $\ms{D}$-modules of Berthelot:
Take an embedding $X\hookrightarrow\ms{P}$ where $\ms{P}$ is
a proper smooth formal scheme over $R$. Then $D(X)$ is a full
subcategory of $D^{\mr{b}}_{\mr{coh}}(\DdagQ{\ms{P}})$ satisfying
some finiteness condition called the overholonomicity and support
condition. See \cite[1.1.1]{A} for more details.

The category $D(X)$ is equipped with a t-structure, called the holonomic
t-structure, whose heart is denoted by $\mr{Hol}(X)$. Philosophically,
this category corresponds to the category of perverse sheaves in the
$\ell$-adic theory. Furthermore, $D(X)$ is equipped with 6 functors.
The category $D(X)$ is a closed monoidal category, so we have 2 functors
$\otimes$ and $\shom$.
The unit object is denoted by $L_X$. When $X$ is smooth, $L=K$, and
quasi-projective, $L_X$ is, in fact, represented by the structure sheaf
up to some shift. Given a morphism $f\colon X\rightarrow Y$ between
schemes of finite type, we have 4 more functors:
\begin{equation*}
 f_*, f_!\colon D(X)\rightarrow D(Y),\quad
   f^*, f^!\colon D(Y)\rightarrow D(X).
\end{equation*}
We denote by $f_*$ and $f^*$ for normal pushforward and pullback in
accordance with the $\ell$-adic theory, and not $f_+$, $f^+$ as in
$\cite{A}$. These functors enjoy a lot of standard properties.
Some of the properties are summarized in \cite[1.1.3]{A}, so we do not
recall here. Finally, exclusively in \ref{recdfnfun}, we consider
Frobenius structure. In order to consider this extra structure, we
remark that ``arithmetic base tuple'' (cf.\ \cite[1.4.10,
2.4.14]{A}) should be fixed, which contains some more information than
geometric base tuple.
We do not go into detail here.

\begin{rem*}
 Since \cite[1.1.3]{A} is written only for realizable schemes, let us
 point out where in the paper the corresponding claim for separated
 schemes of finite type can be found. The functors $f_*$, $f^*$ are
 defined in 2.3.7 and 2.3.10. The functor $\otimes$ is defined in
 2.3.14, and Proposition 2.3.15 implies the existence of $\shom$.
 The functor $f_!$ is defined in 2.3.21, and $f^!$ in 2.3.32.
 The coincidence of $f_!$ and $f_*$ when $f$ is proper follows by
 construction, and the base change is checked in 2.3.22.
 The projection formula is in 2.3.35, the K\"{u}nneth formula is in
 2.3.36, and the localization sequence is in 2.2.9. Duality results
 as well as trace formalism are also written in 2.3.
\end{rem*}

\subsection{}
A (filtered) projective system $``\invlim"_{i\in I}X_i$
is said to be {\em affine \'{e}tale} if all the morphism $X_i\rightarrow
X_j$ are affine and \'{e}tale. By [EGA IV, Proposition 8.2.3], the
projective limit is representable in the category of schemes over $k$.
Let $\mr{Sch}^{\mr{ft}}(k)$ be the category of schemes {\em
separated} of finite type over $k$. We denote by $\mr{Sch}(k)$ the full
subcategory of {\em noetherian} schemes over $k$ which can be written as
the projective limit of an affine \'{e}tale inductive system in
$\mr{Sch}^{\mr{ft}}(k)$.
From now on, we always mean an object of $\mr{Sch}(k)$ by simply saying
schemes. In particular, schemes are assumed noetherian.

\begin{lem}
 \begin{enumerate}
  \item Any scheme in $\mr{Sch}(k)$ is separated.
       
  \item Let $S\in\mr{Sch}(k)$, and $X\rightarrow S$ be a morphism of
	finite type. Then $X\in\mr{Sch}(k)$ as well.

  \item The category $\mr{Sch}(k)$ is closed under taking henselization
	(resp.\ strict henselization).
 \end{enumerate}
\end{lem}
\begin{proof}
 The first claim follows since, writing $X=\invlim X_i$ with
 $X_i\in\mr{Sch}^{\mr{ft}}(k)$, $X_i$ is assumed separated and
 $X\rightarrow X_i$ is affine. The second claim is [EGA IV, 8.8.2].
 For the last claim, we only need to check that the henselization and
 the strict henselization of a point of a noetherian scheme are
 noetherian, but these are [EGA IV, 18.6.6, 18.8.8].
\end{proof}

\subsection{}
Let us introduce the triangulated category of arithmetic
$\ms{D}$-modules on the schemes in $\mr{Sch}(k)$.
Let $X\in\mr{Sch}(k)$. By definition, we may write $X\cong\invlim_{i\in
I} X_i$ where $X_i\in\mr{Sch}^{\mr{ft}}(k)$ and $``\invlim"_{i\in I}X_i$
is affine \'{e}tale. Let $i\rightarrow j$ in
$I$. Since the induced morphism $\phi\colon X_i\rightarrow X_j$ is
\'{e}tale, we have the isomorphism between pull-back functors
$\phi^*\cong\phi^!\colon D(X_j)\rightarrow D(X_i)$.
We define
\begin{equation*}
 D(X):=2\mbox{-}\indlim_{i\in I} D(X_i).
\end{equation*}
Since $\phi^*$ is t-exact with respect to the t-structure, $D(X)$ is
also equipped with a t-structure, whose heart is still denoted by
$\mr{Hol}(X)$.
This category is independent of the choice of projective system up to
canonical isomorphism, which justifies the notation $D(X)$.
Now, assume given {\em any} morphism $f\colon X\rightarrow Y$ in
$\mr{Sch}(k)$.
Then we can find a morphism of affine \'{e}tale projective systems
$``\invlim"_{i\in I}X_i\rightarrow``\invlim"_{j\in J}Y_j$ in
$\mr{Sch}^{\mr{ft}}(k)$ which converges to $f$.
This presentation makes it possible to extend the pull-back and
extraordinary pull-back functor on $\mr{Sch}^{\mr{ft}}(k)$ to
\begin{equation*}
 f^*, f^!\colon D(Y)\rightarrow D(X).
\end{equation*}
Independence of presentation follows easily.
Assume further that $f$ is of finite type. 
Then [EGA IV, 8.8.2] implies that by changing the projective system
$``\invlim"_{i\in I}X_i$ in $\mr{Sch}^{\mr{ft}}(k)$ if necessarily, we
may assume that $I=J$ and that for any $i\rightarrow j$ in $J$ the
following diagram is cartesian:
\begin{equation*}
 \xymatrix{
  X_i\ar[r]^-{\phi_X}\ar[d]_{f_i}\ar@{}[rd]|\square&
  X_j\ar[d]^{f_j}\\
 Y_i\ar[r]^-{\phi_Y}&Y_j.
  }
\end{equation*}
Since $\phi_Y$ is \'{e}tale, we have the canonical isomorphisms
$\phi^*_{Y}\circ f_{j*}\cong f_{i*}\circ\phi_X^*$,
$f_{i!}\circ\phi_X^*\cong \phi^*_{Y}\circ f_{j!}$ by base change and
$\phi_{\star}^*\cong\phi_{\star}^!$.
Thus, we have the push-forward and extraordinary push-forward functor
\begin{equation*}
 f_*, f_!\colon D(X)\rightarrow D(Y).
\end{equation*}

\begin{lem}
 \label{noearhol}
 For a scheme $X$, $\mr{Hol}(X)$ is noetherian and artinian category.
\end{lem}
\begin{proof}
 Let $U$ be a scheme, and write $U\cong\indlim U_i$ where
 $U_i\in\mr{Hol}(U)$. Assume $U_i$ is smooth.
 We say that $\ms{F}\in\mr{Hol}(U)$ is {\em smooth} if there exists
 $i\in I$ and $\ms{F}'\in\mr{Hol}(U_i)$ whose pullback is $\ms{F}$ such
 that $\ms{F}'$ is smooth on $U_i$ in the sense of \cite[1.3.1]{A}.
 It is easy to check that any smooth objects is of finite length, and
 for any $\ms{F}\in\mr{Hol}(X)$, there exists a smooth open subscheme
 $U\subset\mr{Supp}(\ms{F})$ such that $\ms{F}|_U$ is smooth on $U$.
 Now, let $j\colon V\hookrightarrow X$ be an open immersion.
 It suffices to check that for an irreducible object
 $\ms{F}_V\in\mr{Hol}(V)$, $j_{!*}(\ms{F}_V)$ remains to be
 irreducible.
 The verification is standard (see, for example, \cite[Proposition
 1.4.7]{AC}).
\end{proof}

\subsection{}
\label{recdfnfun}
{\em Exclusively in this paragraph, we consider Frobenius structure for
the future reference}. The reader who does not need to consider
Frobenius structure may simply ignore Tate twists appearing in this
paragraph.

Let $\pi\colon X\rightarrow\mb{A}^1_k$ be a morphism of finite type.
Then the exact functors
\begin{equation*}
 \Psi^{\mr{un}}_\pi,\Phi^{\mr{un}}_\pi\colon
  \mr{Hol}(X)\rightarrow\mr{Hol}(X_0)
\end{equation*}
are defined. Let us now recall the definition briefly.
We put $\mc{O}_{\mb{G}_{\mr{m}}}:=\mc{O}_{\widehat{\mb{P}}^1_R,\mb{Q}}
(^\dag\{0,\infty\})$. We define $\slog^n$ for an integer $n\geq0$ as
follows:
\begin{equation*}
 \slog^n:=\bigoplus_{k=0}^{n-1}
  \mc{O}_{\mb{G}_{\mr{m}}}\cdot\log_t^{[k]},
\end{equation*}
the free $\mc{O}_{\mb{G}_{\mr{m}}}$-module of rank $n$ generated by
the symbols $\log_t^{[k]}$. For the later use, we denote
$k!\cdot\log_t^{[k]}$ by $\log_t^k$.
There exists a unique $\DdagQ{\widehat{\mb{P}}^1}$-module structure on
$\slog^n$ so that for $k\geq0$ and $g\in\mc{O}_{\mb{G}_{\mr{m}}}$,
\begin{equation*}
 \partial_t(g\cdot \log_t^{[k]})=\partial_t(g)\cdot\log_t^{[k]}+
  (g/t)\cdot\log_t^{[k-1]},
\end{equation*}
where $\log_t^{[j]}:=0$ for $j<0$.
There is a canonical Frobenius structure on $\slog^n$.
This defines an object of $\mr{Hol}(\mb{A}^1)$ when $L=K$. If
$L\supsetneq K$, we simply extend the scalar.
We have the following exact sequence:
\begin{equation*}
 0\rightarrow\slog^n\rightarrow\slog^{n+m}\rightarrow
  \slog^m(-n)\rightarrow0,
\end{equation*}
where the first homomorphism sends $\log_t^{[i]}$ to $\log_t^{[i]}$ and
the second sends $\log_t^{[i]}$ to $\log_t^{[i-n]}$.
We follow the easy-to-describe definitions of various functors of
Beilinson (cf.\ \cite[Remark 2.6 (i)]{AC2}).

Recall we are given $\pi\colon X\rightarrow\mb{A}^1$, and put $j\colon
X\setminus X_0\hookrightarrow X$, the open immersion, and $i\colon
X_0\hookrightarrow X$, the closed immersion.
Now, we put $\slog^n_\pi:=\pi^*\slog^n$.
Using this we define
\begin{equation*}
 \Pi^{0,i}_{!*}:=
  \indlim_n\,\mr{Ker}\bigl(j_!(\ms{F}\otimes\slog^{n+i}_\pi)(i-1)
  \rightarrow
  j_*(\ms{F}\otimes\slog^{n}_\pi)(-1)\bigr),
\end{equation*}
and put $\Psi^{\mr{un}}_\pi:=\Pi^{0,0}_{!*}(1)$,
$\Xi_\pi:=\Pi^{0,1}_{!*}$.
A key result of \cite[Lemma 2.4]{AC2} is that these limits are
representable in $\mr{Hol}(X)$. We have the following complex in
$\mr{Hol}(X)$:
\begin{equation*}
 j_!\ms{F}\rightarrow\Xi_\pi(\ms{F})\oplus\ms{F}\rightarrow
  j_*\ms{F},
\end{equation*}
and define $\Phi^{\mr{un}}_\pi(\ms{F})$ to be the cohomology of this
complex. 
Here, the homomorphism $j_!\ms{F}\rightarrow\Xi_\pi(\ms{F})$ is the
obvious one, and $\Xi_\pi(\ms{F})\rightarrow j_*\ms{F}$ is the inductive
limit of the connecting homomorphism of the following diagram, recalling
$\slog^1_\pi\cong L_{X\setminus X_0}$:
\begin{equation*}
 \xymatrix{
  &0\ar[r]\ar[d]&j_!(\ms{F}\otimes\slog_\pi^{n+1})\ar@{=}[r]\ar[d]&
  j_!(\ms{F}\otimes\slog_\pi^{n+1})\ar[d]\ar[r]&0\\
 0\ar[r]&
  j_*(\ms{F}\otimes\slog_\pi^{1})
  \ar[r]
  &
  j_*(\ms{F}\otimes\slog_\pi^{n+1})
  \ar[r]
  &
  j_*(\ms{F}\otimes\slog_\pi^{n})(-1)
  \ar[r]
  &
  0.
  }
\end{equation*}
Moreover, this diagram induces the exact sequence
$0\rightarrow\Psi^{\mr{un}}(\ms{F})\rightarrow\Xi_\pi(\ms{F})
\rightarrow j_*(\ms{F})\rightarrow0$, where the surjectivity of the last
homomorphism is also a part of the key result of \cite{AC2}.
This short exact sequence together with the definition of the vanishing
cycle functor yield the following fundamental exact triangle:
\begin{equation}
 \label{fundexunip}
 i^*[-1]\rightarrow\Psi^{\mr{un}}_\pi
  \rightarrow\Phi^{\mr{un}}_\pi\xrightarrow{+1}.
\end{equation}

\begin{rem*}
 \begin{enumerate}
  \item In \cite{AC2}, the object $\mc{I}^{a,b}$ is used instead of
	$\slog^n$. We may check easily that there exists an isomorphism
	$\mc{I}^{a,b}\xrightarrow{\sim}\slog^{b-a}$ where $s^lt^s$,
	using the notation of \cite[2.3]{AC2}, is sent to
	$\log_t^{[l-a-1]}$.
	The embedding $\slog_t^n\hookrightarrow\slog_t^{n+1}$ is
	compatible with the embedding
	$\mc{I}^{a,b}\hookrightarrow\mc{I}^{a,b+1}$.
	The description using $\mc{I}^{a,b}$ is convenient to understand
	the relation with the dual functor, but in order to prove the
	theorem below, $\slog^n$ description reduces notation.
	
  \item We defined $\Psi^{\mr{un}}$ as $\Pi_{!*}^{0,0}(1)$, but in
	\cite{AC2}, following Beilinson, we did not put this Tate twist
	in the definition.
	This Tate twist is put in order that no Tate twist appears in
	(\ref{fundexunip}). Since we do not consider Frobenius structure
	from the next paragraph, we may forget this confusing Tate
	twists.
 \end{enumerate}
\end{rem*}

\subsection{}
Now, let $S$ be a scheme of finite type over $k$, and $s\in S$ be a
regular point of codimension $1$.
Let $\pi\colon X\rightarrow S$ be a morphism of finite type.
For a morphism $h\colon S\rightarrow\mb{A}^1$ such that $h(s)=0$, the
functor $\Psi^{\mr{un}}_{h\circ\pi}$ is defined.
In this paper, for a morphism of schemes $f\colon X\rightarrow Y$ and a
point $y\in Y$, the fiber $X\otimes_Yk(y)$ is denoted by $X_y$.

\begin{thm*}
 The functors $\Psi^{\mr{un}}_{h\circ\pi}|_{X_s}$,
 $\Phi^{\mr{un}}_{h\circ\pi}|_{X_s}$
 does not depend on the choice of $h$ up to canonical equivalence.
 This justifies to denote these functors by $\Psi^{\mr{un}}_\pi$ and
 $\Phi^{\mr{un}}_\pi$ respectively.
\end{thm*}
\begin{proof}
 Let $\mb{A}^2_{(x,y)}\rightarrow\mb{A}^1_t$ be the morphism sending
 $t$ to $xy$. On $\mb{A}^2$, we construct a homomorphism
 \begin{equation*}
  \alpha\colon
   \slog_{xy}^{n}\rightarrow\slog_{x}^{n}\otimes\slog_y^n.
 \end{equation*}
 by sending $(\log uh)^k$ to $\sum_{i=0}^k\binom{k}{i}(\log
 u)^i\otimes(\log h)^{k-i}$. It is easy to check that this defines a
 homomorphism of $\ms{D}^\dag$-modules.
 Now, shrink $S$ around $s$, which is allowed since we only need the
 equivalence after $|_{X_s}$,
 so that the closure of $s$, is a smooth divisor denoted by $D$.
 Let $u,v\in\mc{O}_S$. These functions define a morphism
 $\rho\colon S\rightarrow\mb{A}^2$ by sending $x$, $y$ to $u$, $v$
 respectively. Then we get a homomorphism in $\mr{Hol}(S)$
 \begin{equation*}
  \alpha_{u,v}:=\rho^*(\alpha)\colon
   \slog_{uv}^n\rightarrow\slog_u^n\otimes\slog_v^n.
 \end{equation*}
 Given $u,v,w\in\mc{O}_S$, the following diagram is commutative:
 \begin{equation}
  \label{multlogsh}
  \xymatrix{
   \slog^n_{uvw}\ar[r]\ar[d]&
   \slog^n_{u}\otimes\slog_{vw}\ar[d]\\
  \slog^n_{uv}\otimes\slog^n_{w}\ar[r]&
   \slog^n_{u}\otimes\slog^n_{v}\otimes\slog^n_{w}.
   }
 \end{equation}

 Now, for $h\colon S\rightarrow\mb{A}^1$, we denote
 $\Psi^{\mr{un}}_{h\pi}$, $\Phi^{\mr{un}}_{h\pi}$,
 $\Xi^{\mr{un}}_{h\pi}$ by $\Psi_h$, $\Phi_h$, $\Xi_h$ respectively.
 Take $u\in\mc{O}_S^{\times}$.
 Because the image of the associated morphism $u\colon
 S\rightarrow\mb{A}^1$ is contained in $\mb{G}_{\mr{m}}$, the object
 $\slog_u^n$ is an iterated extension of the trivial object $L_S$.
 This implies that for $\star\in\{!,*\}$,
 \begin{equation}
  \label{smcommjph}
   j_{\star}\bigl(\ms{F}\otimes\slog_{u\pi}^n
   \otimes\slog_{h\pi}^n\bigr)
   \cong
   j_{\star}(\ms{F}\otimes\slog_{h\pi}^n)
   \otimes\slog_{u\pi}^n,
 \end{equation}
 where $j\colon X\setminus(h\circ\pi)^{-1}(0)\hookrightarrow X$.
 Define
 \begin{equation*}
   \Pi^{0,i}_{(u,h)!*}(\ms{F}):=
    \indlim_n\,\mr{Ker}\bigl(
    j_!(\ms{F}\otimes\slog_{u\pi}^{n+i}\otimes
    \slog_{h\pi}^{n+i})
    \rightarrow
    j_*(\ms{F}\otimes\slog_{u\pi}^{n+i}\otimes
    \slog_{h\pi}^{n})
    \bigr),
 \end{equation*}
 and put $\Psi_{u,h}:=\Pi^{0,0}_{(u,h)!*}$,
 $\Xi_{u,h}:=\Pi^{0,1}_{(u,h)!*}$ as usual.
 Then, (\ref{smcommjph}) induces the canonical isomorphisms
 \begin{equation*}
  \Xi_{u,h}\cong \bigl(\indlim_{n}\slog^n_{u\pi}\bigr)
   \otimes\Xi_h,\quad
  \Psi_{u,h}\cong \bigl(\indlim_{n}\slog^n_{u\pi}|_D\bigr)
   \otimes\Psi_h
 \end{equation*}
 as Ind objects. We have a homomorphism
 \begin{equation*}
  \slog^n_{u}|_D\cong\bigoplus_{k=0}^{n-1}
   L_D\cdot\log_u^{[k]}
   \xrightarrow{\mr{pr}_0}L_D,
 \end{equation*}
 where $\mr{pr}_0$ denotes the projection by the factor indexed by
 $\log^{[0]}_u$.
 By taking the limit, we have a homomorphism
 $\mr{pr}_0\colon\indlim_{n}\slog^n_{u}|_D\rightarrow L_D$.
 Composing everything, we have
 \begin{equation*}
  \phi_{h,uh}\colon
   \Psi_{uh}\xrightarrow{\alpha_{u,v}}
   \Psi_{u,h}\cong
   \bigl(\indlim_{a}\slog^n_{u}|_D\bigr)
   \otimes\Psi_h
   \xrightarrow{\mr{pr}_0}
   \Psi_h.
 \end{equation*}
 We may check easily that $\phi_{h,h}=\mr{id}$.
 Using (\ref{multlogsh}), it is also an easy exercise to show that
 $\phi_{vuh,uh}\circ\phi_{uh,h}=\phi_{vuh,h}$ for
 $v\in\mc{O}_S^{\times}$.
 Thus, $\phi_{uh,h}$ is an isomorphism for any $u\in\mc{O}_S^{\times}$.
 In order to show the theorem for $\Phi^{\mr{un}}$, we define
 $\Phi_{u,h}$ by the cohomology of the following complex
 \begin{equation*}
  j_!\ms{F}\otimes\indlim_{n}\slog^n_{u\pi}
   \rightarrow\Xi_{u,h}(\ms{F})\oplus
   \bigl(\ms{F}\otimes\indlim_{n}\slog^n_{u\pi}\bigr)
   \rightarrow
  j_*\ms{F}\otimes\indlim_{n}\slog^n_{u\pi},
 \end{equation*}
 and argue similarly.
\end{proof}

\subsection{}
Let $X\xrightarrow{\pi}S\xrightarrow{h}S'$ be morphisms of schemes of
finite type over $k$, $h$ is dominant, and $s\in S$, $s'\in S'$ be
codimension $1$ regular points such that $s$ is sent to $s'$. We have
the canonical morphism $h'\colon X_s\rightarrow X_{s'}$. Then by the
construction of nearby/vanishing cycle functors, we have
\begin{equation}
 \label{invbcun}
 h'^*\Psi^{\mr{un}}_{h\circ\pi}\cong\Psi^{\mr{un}}_{\pi},\quad
  h'^*\Phi^{\mr{un}}_{h\circ\pi}\cong\Phi^{\mr{un}}_{\pi}.
\end{equation}

By saying $(S,s,\eta)$ is a {\em henselian trait}, we mean $S$ is the
spectrum of a henselian discrete valuation ring with closed point $s$
and generic point $\eta$.
Let $(S,s,\eta)$ be a henselian trait. Assume given a morphism
$\pi\colon X\rightarrow S$. Even if $X$ and $S$ are not of finite type
over $k$, we may define exact functors
\begin{equation*}
 \Psi^{\mr{un}}_\pi, \Phi^{\mr{un}}_\pi\colon
  \mr{Hol}(X)\rightarrow\mr{Hol}(X_s).
\end{equation*}
Indeed, we can find a diagram
\begin{equation*}
 \xymatrix{
  X\ar[r]^-{\rho'}\ar@/_10pt/[rd]_{\pi}&
  \mc{X}\times_{\mc{S}}S\ar[d]\ar[r]\ar@{}[rd]|\square&
  \mc{X}\ar[d]^{\widetilde{\pi}}\\
 &S\ar[r]^-{\rho}&\mc{S}.
  }
\end{equation*}
Here, $\widetilde{\pi}$ is a morphism of schemes of finite type over
$k$,
$\rho(s)$ is a regular point $s'\in\mc{S}$ of codimension $1$, $\rho'$
is the limit of a projective system of affine \'{e}tale
$\mc{X}\times_{\mc{S}}S$-schemes.
Let $\rho_X\colon X_s\rightarrow\mc{X}_{s'}$. Then we define
$\Psi^{\mr{un}}_{\pi}:=\rho_X^*\Psi^{\mr{un}}_{\widetilde{\pi}}$,
$\Phi^{\mr{un}}_{\pi}:=\rho_X^*\Phi^{\mr{un}}_{\widetilde{\pi}}$.
The equivalence (\ref{invbcun}) implies that these do not depend on the
choice of the diagram.

\subsection{}
Now, let $(S,s,\eta)$ be a {\em strict} henselian trait, and we define
the category $\mr{Hen}(S)$ to be the category of henselian traits over
$S$ which is generically \'{e}tale. Morphisms are $S$-morphisms.
Given a morphism $h\colon S'\rightarrow S$, consider the following
commutative diagram:
\begin{equation*}
 \xymatrix{
  s'\ar[r]^-{i'}\ar[d]_-{h_s}^{\sim}&
  S'\ar[d]^{h}\ar@{}[rd]|\square&
  \eta'\ar[l]_-{j'}\ar[d]^-{h_\eta}\\
 s\ar[r]^-{i}&S&\eta.\ar[l]_-{j}
  }
\end{equation*}
Now, assume given a morphism $f\colon X\rightarrow S$.
Let $f'\colon X\times_SS'\rightarrow S'$.
By abuse of notation, we use the same symbols for the base change from
$S$ to $X$, for example $h\times\mr{id}\colon X\times_SS'\rightarrow X$
is denoted by $h$. We have a canonical morphism
\begin{equation*}
 \Psi^{\mr{un}}_{f}\rightarrow
  h_{s,*}\circ\Psi^{\mr{un}}_{f'}\circ h^*.
\end{equation*}
Since
$\Psi^{\mr{un}}_{f'}\circ h^*\cong\Psi^{\mr{un}}_{f'}\circ
(j'_*h_{\eta}^*j^*)$, we have
$\H^i\bigl(\Psi^{\mr{un}}_{f'}\circ h^*(\ms{F})\bigr)=0$
for any $\ms{F}\in\mr{Hol}(X)$ and $i\neq0$.
Now, we have the following diagram of exact sequences by
(\ref{fundexunip}).
\begin{equation*}
 \xymatrix{
  &
  0\ar[r]\ar[d]&
  \H^{-1}i^*(\ms{F})
  \ar[r]\ar[d]^{\sim}&
  \Psi^{\mr{un}}_f(\ms{F})\ar[d]
  \\
 0\ar[r]&
  \H^{-1}\bigl(h_{s*}\Phi^{\mr{un}}_f(h^*\ms{F})\bigr)
  \ar[r]&
  \H^{-1}\bigl(h_{s*}i'^*(h^*\ms{F})\bigr)
  \ar[r]&
  h_{s*}\Psi^{\mr{un}}_f(h^*\ms{F})
  }
\end{equation*}
Here the middle homomorphism is an isomorphism since $h_s$ is an
isomorphism. This implies that
$\H^{-1}\bigl(h_{s*}\Phi^{\mr{un}}_f(h^*\ms{F})\bigr)=0$.
Together with the other parts of the diagram of long exact sequences, we
have
$\H^{i}\bigl(h_{s*}\Phi^{\mr{un}}_f(h^*\ms{F})\bigr)=0$ for $i\neq0$.
This enables us to define exact functors
$\mr{Ind}\mr{Hol}(X)\rightarrow\mr{Ind}\mr{Hol}(X_s)$ as follows:
\begin{align*}
 \Psi_{f}:=\indlim_{S'\in\mr{Hen}(S)}
  h_{s,*}\circ\Psi^{\mr{un}}_{f'}\circ h^*,\quad
  \Phi_{f}:=\indlim_{S'\in\mr{Hen}(S)}
  h_{s,*}\circ\Phi^{\mr{un}}_{f'}\circ h^*.
\end{align*}
These are endowed with an action of
$I:=\mr{Gal}(\eta^{\mr{sep}}/\eta)$ where $\eta^{\mr{sep}}$ is the
separable closure of $\eta$.

\begin{prop}
 \label{dim0case}
 Let $(S,s,\eta)$ and $(S',s',\eta')$ be strict henselian traits, and
 $\pi\colon S'\rightarrow S$ be a dominant morphism of finite type.
 \begin{enumerate}
  \item For $\ms{F}\in\mr{Hol}(\eta')$, the nearby cycles
	$\Psi_\pi(\ms{F})$ is representable in $\mr{Hol}(s')$.
  \item For $\ms{F}\in\mr{Hol}(\eta')$, we have
	$\mr{rk}(\ms{F})=\mr{rk}(\Psi_\pi(\ms{F}))$.
  \item\label{invbcphi}
       The morphism
       $\pi_s^*\circ\Psi_{\mr{id}}\xrightarrow{\sim}
       \Psi_{\mr{id}}\circ\pi^*$, where
       $\pi_s\colon s'\xrightarrow{\sim}s$
       is the induced morphism, is isomorphic.
 \end{enumerate}
\end{prop}
\begin{proof}
 Let us check \ref{invbcphi}. Assume $f$ is generically \'{e}tale.
 Then $S'\in\mr{Hen}(S)$. Thus we have the functor
 $\mr{Hen}(S')\rightarrow\mr{Hen}(S)$ sending $S''\rightarrow S'$ to
 $S''\rightarrow S'\rightarrow S$. This functor is cofinal, and since
 $\pi_s$ is an isomorphism, we get the claim in this case.
% and we have the following diagram:
% \begin{equation*}
%  \xymatrix@R=3pt{
%   &&&\overline{\eta}\ar[ld]\ar[lddd]\\
%  s'\ar[r]\ar[dd]^{\sim}_{\pi_s}&
%   S'\ar[dd]^{\pi}\ar@{}[rdd]|\square&
%   \eta'\ar[dd]\ar[l]&\\
%  &&&\\
%  s\ar[r]&S&\eta.\ar[l]&}
% \end{equation*}
% Since the vertical morphism on the left hand side is an isomorphism, we
% get the claim.
 If $f$ is not generically \'{e}tale, the morphism
 $\eta'\rightarrow\eta$ breaks up into
 $\eta'\xrightarrow{a}\eta^{\mr{sep}}\xrightarrow{b}\eta$ where $a$ is
 purely inseparable and $b$ is \'{e}tale. By the generically \'{e}tale
 case we have already treated, it suffices to check the claim for the
 case where $f_\eta$ is purely inseparable. In this case, $f$ is
 universally homeomorphic, thus the functors do not see the difference
 between $S$ and $S'$ (cf.\ \cite[Lemma 1.1.3]{A}).

 Let us check the other two claims. By \ref{invbcphi}, we may assume
 that $\pi=\mr{id}$.
 Let $Y$ be a smooth scheme and $D$ be a smooth divisor with the generic
 point $\eta_D$ such that the strict henselization of $Y$ at $\eta_D$ is
 $S$. By Kedlaya's semistable reduction theorem
 \cite{K}, at the cost of shrinking $Y$ further, there exists a
 surjective morphism $g\colon Y'\rightarrow Y$ such that $Y'$ is smooth,
 $g$ is \'{e}tale outside of $D$, and the pull-back of $\ms{F}$ to $Y'$
 is log-extendable along the smooth irreducible divisor
 $D':=g^{-1}(D)$ with generic point $\eta_{D'}$.
 Replacing $Y$ by its \'{e}tale neighborhood around $\eta_D$, we may
 assume that $\eta_{D'}\rightarrow\eta_D$ is an isomorphism.
 Take any function $h\in\mc{O}_{Y'}$ such that the zero-locus is equal
 to $D'$. In this situation, the computation of \cite[Lemma
 2.4]{AC2} shows that $\Psi^{\mr{un}}_{h}(g^*\ms{F})$ is smooth
 object on $D'$ of rank equal to that of $\ms{F}$.
 Now, take a finite morphism $\alpha\colon Y''\rightarrow Y'$ which is
 \'{e}tale outside of $D'$. Then $g^*\ms{F}|_{Y'\setminus D'}$ is a direct
 factor of $\alpha_*\alpha^*g^*\ms{F}|_{Y'\setminus D'}$,
 so we get that the canonical homomorphism
 \begin{equation}
  \label{splinjcom}
  \Psi_h^{\mr{un}}(g^*\ms{F})\rightarrow
   \alpha_*\Psi_{h\circ\alpha}^{\mr{un}}(\alpha^*g^*\ms{F})
 \end{equation} 
 is a split injective homomorphism.
 Since $\alpha^*g^*\ms{F}$ is also log-extendable, the rank is the same
 as that of $\ms{F}$. This implies that the canonical map
 (\ref{splinjcom}) is an isomorphism. Thus,
 $\Psi^{\mr{un}}_{\mr{id}}(\ms{F})$ is the same as the pull-back of
 $g_*\Psi^{\mr{un}}_h(\ms{F})$.
\end{proof}

\section{Finiteness of nearby cycle}
\label{finthm}
Throughout this section, we fix a strict henselian trait $(S,s,\eta)$.

\subsection{}
First, we need a preparation.
Let $f\colon X\rightarrow Y$ be a morphism of finite type.
We put $f^i_*:=\H^if_*\colon\mr{Hol}(X)\rightarrow\mr{Hol}(Y)$.
This extends canonically to Ind-categories (cf.\ \cite[1.2]{A}), and
defines a functor $\mr{Ind}\mr{Hol}(X)\rightarrow\mr{Ind}\mr{Hol}(Y)$,
which we still denote by $f^i_*$.
If $f$ is smooth of codimension $d$, we have the exact functor
$f^*[d]\colon\mr{Hol}(Y)\rightarrow\mr{Hol}(X)$.
This functor also extends to Ind-categories, and is denoted by
$f^{\circledast}$.

\begin{lem*}
 \label{fundcommind}
 Let $S$ be a strict henselian trait, and $f\colon X\rightarrow Y$ be a
 morphism of $S$-schemes.
 \begin{enumerate}
  \item If $f$ is proper, then there exists a canonical isomorphism
	$\Psi\circ f^i_*\cong f^i_*\circ\Psi$ in $\mr{Ind}\mr{Hol}(Y)$
	for each $i$.
  \item If $f$ is smooth, then there exists a canonical isomorphism
	$f^{\circledast}\circ\Psi\cong\Psi\circ f^{\circledast}$
	in $\mr{Ind}\mr{Hol}(X)$.
 \end{enumerate}
\end{lem*}
\begin{proof}
 These are exercise of six-functor formalism, so we leave the
 verification to the reader.
\end{proof}

\begin{thm}
 \label{finthm}
 If $\pi\colon X\rightarrow S$ is of finite type, the functors
 $\Psi_\pi$ and $\Phi_\pi$ define functors
 from $\mr{Hol}(X)$ to $\mr{Hol}(X_s)$.
\end{thm}
By the fundamental exact triangle
$\Psi_\pi^{\mr{un}}\rightarrow\Phi_\pi^{\mr{un}}\rightarrow
i^*\xrightarrow{+}$, where $i\colon X_s\rightarrow X$,
it suffices to check the theorem just for
$\Psi_\pi$. The idea of the proof is essentially the same as
\cite[Th\'{e}or\`{e}me 3.2]{DF}.
The proof is divided into several parts.
We prove the theorem by induction on the dimension of $X_\eta$.
The case where $\dim(X_\eta)=0$ has already been treated in Proposition
\ref{dim0case}.
We assume that the theorem holds for $X$ such that $\dim(X_\eta)<n$.
From now on, we assume that $\dim(X_\eta)=n$.

\begin{lem}[{\cite[Lemme 3.5]{DF}}]
 \label{finprojle}
 Let $K$ be a field containing $k$.
 Let $X\subset \mb{A}^n_K$ be a closed subscheme,
 $\ms{F}\in\mr{Ind}\mr{Hol}(X)$, and $\overline{\eta}$ be a
 geometric generic
 point of $\mb{A}^1_K$. Let $X_{\overline{\eta},i}$ be the geometric
 generic fiber of the morphism
 $X\subset\mb{A}^n\xrightarrow{\mr{pr}_i}\mb{A}^1$, where $\mr{pr}_i$
 denotes the $i$-th projection, and $\ms{F}_{\overline{\eta},i}$ denotes
 the pull-back of $\ms{F}$ to $X_{\overline{\eta},i}$. Assume
 $\ms{F}_{\overline{\eta},i}\in\mr{Hol}(X_{\overline{\eta},i})$ for all
 $i$. Then there exists $\ms{F}'\in\mr{Hol}(X)$ contained in $\ms{F}$
 such that the local sections of $\ms{F}/\ms{F}'$ are supported on
 finitely many points, namely it is isomorphic to
 $\bigoplus_{x\in|\mb{A}^n|}\ms{G}_x$ where $\ms{G}_x$ is supported on
 $x$.
\end{lem}
\begin{proof}
 Let $\pi\colon\mb{A}^n\rightarrow\mb{A}^1$ be a projection. For each
 $x\in\mb{A}^1$, denote by
 $i_x\colon\mb{A}_x:=\pi^{-1}(x)\hookrightarrow\mb{A}^n$ and by $\eta$
 the generic point of $\mb{A}^1$. Assume given
 $\ms{G}\in\mr{Ind}\mr{Hol}(\mb{A}^n)$ such that the pull-back to
 $\mb{A}_\eta$ is $0$. Then
 \begin{equation*}
  \ms{G}\cong\bigoplus_{x\in|\mb{A}^1|}i_{x,*}i^*_x\ms{G}.
 \end{equation*}
 Indeed, since $\mr{Hol}(\mb{A}^n)$ is a noetherian category by Lemma
 \ref{noearhol}, we may write $\ms{G}\cong\indlim_i\ms{G}_i$ where
 $\ms{G}_i\hookrightarrow\ms{G}$ and $\ms{G}_i\in\mr{Hol}(\mb{A}^n)$.
 Since the pull-back by $\mb{A}_\eta\rightarrow\mb{A}^n$ is exact, by
 assumption, $\ms{G}_i$ becomes $0$ on $\mb{A}_\eta$. This implies that
 $\ms{G}_i$ is supported on $\coprod_{i=1}^n\mb{A}_{x_i}$ where $x_i$
 are closed points of $\mb{A}^1$, and the claim follows.

 Fix $i$.
 There exists an \'{e}tale neighborhood $U$ of $\overline{\eta}$ and
 $\ms{H}_i\in\mr{Hol}(X_{U,i})$, where
 $j_{i}\colon X_{U,i}:=X\times_{\mr{pr}_i,\mb{A}^1}U\hookrightarrow X$,
 such that its pull-back to $X_{\overline{\eta},i}$ is
 $\ms{F}_{\overline{\eta},i}$. By shrinking $U$ if necessarily, we may
 assume that the isomorphism
 $\ms{H}_{i,\overline{\eta},i}\xrightarrow{\sim}
 \ms{F}_{\overline{\eta},i}$ is induced by a homomorphism
 $\ms{H}_i\rightarrow\ms{F}_{U,i}$ where $\ms{F}_{U,i}$ denotes the
 pull-back of $\ms{F}$ to $X_{U,i}$. Now, put
 \begin{equation*}
  \ms{F}':=\sum_{i=1}^n\mr{Im}\bigl(j_{i!}\ms{H}_i\rightarrow\ms{F}\bigr)
   \subset\ms{F}.
 \end{equation*}
 By construction, $\ms{F}'\in\mr{Hol}(X)$ and
 $(\ms{F}/\ms{F}')_{\overline{\eta},i}=0$ for any $n$.
 Thus, using the observation above, any local section of
 $\ms{F}/\ms{F}'$ is supported on finitely many points.
\end{proof}

\subsection{}
\label{difftrcom}
Let $s'$ be the generic point of $\mb{A}^1_s$, and let $(S',s',\eta')$
be the strict henselization of $\mb{A}_S^1$ at the generic point $s'$ of
the divisor $\mb{A}^1_s$. Let $h\colon S'\rightarrow S$ the morphism.
Deligne constructed the following diagram in
the proof of \cite[Lemme 3.3]{DF}:
\begin{equation*}
 \xymatrix{
  \overline{\eta}'\ar[r]^-{P}\ar[rd]&
  \mr{Spec}(k')\ar[r]^-{G}\ar[d]&\eta'\ar[d]\\
 &\overline{\eta}\ar[r]^-{G}&\eta}
\end{equation*}
where horizontal maps are algebraic extensions of fields, and $P$, $G$
denote the Galois groups of the extension. The group $P$ is a
pro-$p$-group. Now, let $\pi'\colon X'\rightarrow S'$ be a morphism, and
let $\ms{F}\in\mr{Hol}(X'_{\eta'})$. Note that
$X'_{\eta}=X'_{\eta'}$ and $X'_s=X'_{s'}$.
Then we have $\Psi_{h\circ\pi'}(\ms{F})\cong\Psi_{\pi'}(\ms{F})^P$.

\subsection{Proof of Theorem \ref{finthm}}\mbox{}\\
Now, let $\pi\colon X\rightarrow S$ be a morphism of finite type.
We first assume $X$ affine, and take a closed immersion
$X\hookrightarrow\mb{A}^n_S$, and let $f\colon
X\subset\mb{A}^n_S\rightarrow\mb{A}^1_S$ where the second morphism is a
projection.
Recall from the previous paragraph that $\lambda\colon
S'\hookrightarrow\mb{A}^1_S$ is a strict henselization of $\mb{A}^1_S$.
Consider the following diagram:
\begin{equation*}
 \xymatrix@C=40pt{
  X'\ar[r]^-{f'}\ar[d]_{\lambda_X}\ar@{}[rd]|\square&
  S'\ar[d]^{\lambda}&\\
 X\ar[r]^-{f}\ar@/_15pt/[rr]_-{\pi}&\mb{A}^1_S\ar[r]&S.
  }
\end{equation*}
Let $\ms{F}'$ be the pull-back of $\ms{F}$ to $X'$. Then we have
\begin{equation*}
 \lambda^{\circledast}_X\Psi_{\pi}(\ms{F})\cong\Psi_{\pi\circ\lambda_X}(\ms{F})
  \cong\Psi_{f'}(\ms{F})^P,
\end{equation*}
where the first isomorphism follows by Lemma \ref{fundcommind},
and the second by \ref{difftrcom}. Now, the induction hypothesis tells us
that $\Psi_{f'}(\ms{F})^P\in\mr{Hol}(X'_{s'})$, and by construction
$X'_{s'}=X'_s$. This implies that
$\lambda^{\circledast}_X\Psi_{\pi}(\ms{F})\in\mr{Hol}(X'_{s})$.
Lemma \ref{finprojle} ensures the existence of $\ms{G}\in\mr{Hol}(X_s)$
contained in $\Psi_\pi(\ms{F})$ such that the local sections of
$\Psi_\pi(\ms{F})/\ms{G}$ are supported on finitely many points.
Now, if $X$ is not affine, we take a finite affine open covering
$\{U_i\}$ and we can get such $\ms{G}_i$ for each subscheme. Then
$\ms{G}:=\mr{Im}\bigl(\bigoplus\H^0\ms{G}_{i,!}\rightarrow
\Psi_\pi(\ms{F})\bigr)$, where $\ms{G}_{i,!}$ denotes the extension by
zero of $\ms{G}_i$ to $X$, is in $\mr{Hol}(X_s)$ and the local sections
of $\Psi_\pi(\ms{F})/\ms{G}$ are supported on finitely many points.

In order to show the finiteness of $\Psi_\pi(\ms{F})$, we may assume $X$
is proper over $S$. Take $\ms{G}$ as above, and we may write
$\Psi_\pi(\ms{F})/\ms{G}\cong\bigoplus_{x\in|X_s|}\ms{G}_x$ where
$\ms{G}_x$ is supported on $x$. We have the following long exact
sequence:
\begin{equation*}
 \xymatrix@R=10pt{
%  \uwave{\pi^i_{s*}(\ms{G})}
  \dots\ar[r]&
  \pi^i_{s*}\bigl(\Psi_\pi(\ms{F})\bigr)\ar[r]\ar@{-}[d]^{\sim}&
  \pi^i_{s*}\bigl(\bigoplus_x\ms{G}_x\bigr)\ar[r]&
  \uwave{\pi^{i+1}_{s*}(\ms{G})}\ar[r]&\dots
  \\
 &\uwave{\Psi_{\mr{id}}(\pi^i_*\ms{F})}&&&&
  }
\end{equation*}
where the vertical isomorphism follows by Lemma
\ref{fundcommind}. Since we already know that the objects with
$\uwave{(-)}$ are in $\mr{Hol}(s)$, we have
$\bigoplus_x\pi_{s*}\ms{G}_x\cong
\pi_{s*}\bigl(\bigoplus_x\ms{G}_x\bigr)\in\mr{Hol}(s)$.
This implies that $\bigoplus_{x\in|X_s|}\ms{G}_x\in\mr{Hol}(X_s)$, and
thus $\Psi_\pi(\ms{F})\in\mr{Hol}(X_s)$ as required.
\qed

\begin{cor}
 \label{fundcomm}
 Let $\pi\colon X\rightarrow S$ be a morphism of finite type, and
 $f\colon X\rightarrow Y$ be a morphism of $S$-schemes of finite type.
 \begin{enumerate}
  \item\label{fundcommpr}
       If $f$ is proper, then there exists a canonical isomorphism
       $\Psi\circ f_*\cong f_*\circ\Psi$ in $D(Y)$.
  \item\label{fundcommsm}
       If $f$ is smooth, then there exists a canonical isomorphism
       $f^*\circ\Psi\cong\Psi\circ f^*$ in $D(X)$.
  \item\label{fundextr}
       We have the exact triangle of functors
       $i^*[-1]\rightarrow\Psi_\pi\rightarrow\Phi_\pi\xrightarrow{+}$.
 \end{enumerate}
\end{cor}

\begin{rem*}
 Let $R\psi$ and $R\phi$ be the nearby and vanishing cycle functor for
 $\ell$-adic sheaves. The exact triangle for nearby/vanishing cycle
 functor usually goes
 $i^*\rightarrow R\psi\rightarrow R\phi\xrightarrow{+}$.
 This difference arises because in the spirit of Riemann-Hilbert
 correspondence, $\Psi=R\psi[-1]$, $\Phi=R\phi[-1]$.
 We could have employed this normalization, but in order to be
 consistent with \cite{AC2}, we decided not to take the shift.
\end{rem*}

\section{Smooth objects}
\label{smsec}
\subsection{}
Let $X$ be a scheme of finite type over $k$.
Then $D(X)$ is equipped with two t-structures; the holonomic t-structure
whose heart is $\mr{Hol}(X)$, and the constructible t-structure defined
in \cite[\S1.3]{A}. The heart of constructible t-structure is
denoted by $\mr{Con}(X)$.
Given a morphism of finite type $f\colon X\rightarrow Y$,
the pull-back $f^*$ is exact with respect to constructible t-structure
by \cite[1.3.4]{A}. Thus constructible t-structure extends to a
t-structure on $D(X)$ for any scheme $X$.
The cohomology functor for holonomic t-structure is denoted
by $\H^*$, as we have already used several times,
and the constructible t-structure by $\cH^*$.

\begin{dfn}
 Let $X$ be a scheme. Then $\ms{F}\in\mr{Con}(X)$
 is said to be {\em smooth} if for any morphism $\phi\colon S\rightarrow
 X$ from a strict henselian trait, $\Phi_{\mr{id}}(\phi^*\ms{F})=0$.
 The full subcategory of smooth objects in $\mr{Con}(X)$ is denoted by
 $\mr{Sm}(X)$.
\end{dfn}

By Theorem \ref{isocandsm} below, we can see that this definition is in
fact a generalization of smoothness defined in \cite[1.3.1]{A}.
To be more precise, when $X$ is a realizable scheme over $k$ such that
$X_{\mr{red}}$ is smooth, $\mr{Sm}(X)$ is the same as the category
introduced in \cite[1.1.3 (12)]{A}.

\begin{lem}
 \label{smpullok}
 Let $f\colon Y\rightarrow X$ be a proper surjective morphism, and
 $\ms{F}\in\mr{Con}(X)$. If $f^*\ms{F}$ is smooth, then $\ms{F}$ is
 smooth.
\end{lem}
\begin{proof}
 Let $f\colon S'\rightarrow S$ be a finite morphism between strict
 henselian traits.
 In this case, we have an isomorphism
 $f_s^*\Psi_{\mr{id}}(\ms{F})\cong\Psi_{\mr{id}}(f^*\ms{F})$ by
 Proposition \ref{dim0case}.\ref{invbcphi}. Since $f_s^*$ is
 an isomorphism, the claim follows.
 Consider the general case. Given a morphism $\phi\colon S\rightarrow X$
 from a strict henselian trait, the fiber $Y\times_X\eta$ is non-empty
 since $f$ is assumed surjective.
 There exists a finite extension $\eta'$ of
 $\eta$ and a morphism $\eta'\rightarrow Y$ compatible with
 $\eta\rightarrow X$, thus, by valuative criterion of properness, we
 have a commutative diagram
 \begin{equation*}
  \xymatrix{
   S'\ar[r]^-{\phi'}\ar[d]_{f'}&Y\ar[d]^{f}\\
  S\ar[r]^-{\phi}&X,}
 \end{equation*}
 where $S'$ is a strict henselian trait, and $f'$ is dominant.
 By the finite morphism case we have already treated, it suffices to
 check that $f'^*\phi^*\ms{F}$ is smooth, which follows by assumption.
\end{proof}

\subsection{}
Let us recall briefly some basics of the theory of descent.
Let $\Delta$ be the category of three objects $[0],[1],[2]$ where
$[i]=\{0,\dots,i\}$. The morphism $[i]\rightarrow[j]$ is a
non-decreasing map. We denote by $\delta^n_j\colon[n-1]\rightarrow[n]$
the map skipping $j$ and $\sigma^n_j\colon[n+1]\rightarrow[n]$ be the
map such that $(\sigma^{n}_j)^{-1}(j)=\{j,j+1\}$.
A {\em simplicial scheme} $X_\bullet$ is a contravariant functor
$\Delta^{\circ}\rightarrow\mr{Sch}(k)$.
Usually, this type of simplicial scheme is called $2$-truncated
simplicial scheme, but since we only use these, we abbreviate the word
``$2$-truncated''. We put $X_i:=X_{\bullet}([i])$, and
$d^n_j:=X_{\bullet}(\delta^n_j)$ and $s^n_j:=X_{\bullet}(\sigma^n_j)$.
The category of descent data for $X_\bullet$ denoted by
$\mr{Con}(X_\bullet)$ consists of the
following data as objects:
\begin{itemize}
 \item an object $\ms{F}\in\mr{Con}(X_0)$;
 \item an isomorphism $\phi\colon
       d^{1*}_0\ms{F}\xrightarrow{\sim}d^{1*}_1\ms{F}$;
\end{itemize}
which satisfies the cocycle condition
$d^{2*}_2(\phi)\circ d^{2*}_0(\phi)=d^{2*}_1(\phi)$ on $X_2$ and
$s^{0*}_0\phi=\mr{id}$ on $X_0$.
Given an augmentation $f\colon X_\bullet\rightarrow X$, namely a
morphism of simplicial schemes considering $X$ as the constant
simplicial scheme, we say that $\mr{Con}(X)$ {\em satisfies the descent}
with respect to $f$ if the canonical functor
$\mr{Con}(X)\rightarrow\mr{Con}(X_\bullet)$ is an equivalence of
categories.
An augmentation $X_{\bullet}\rightarrow X$ is a
{\em proper hypercovering} if the canonical morphisms $X_0\rightarrow
X$, $X_0\times_XX_0\rightarrow X_1$, and $X_2\rightarrow
\mr{cosk}_1\mr{sk}_1(X_{\bullet})_2$ are proper surjective.
For the functors $\mr{sk}_1$ and $\mr{cosk}_1$, one can refer, for
example, to \cite[Tag 0AMA]{Stack}.

\begin{lem}
 \label{descone}
 Let $f\colon Y\rightarrow X$ be a proper surjective morphism. Then
 the hypercovering
 $Y_\bullet:=\mr{cosk}_1(Y\times_XY\rightrightarrows Y)\rightarrow X$
 satisfies the descent.
\end{lem}
\begin{proof}
 We have the natural functor
 $\alpha\colon\mr{Con}(X)\rightarrow\mr{Con}(Y_\bullet)$, and we need to
 show that this is an equivalence.
 Let us construct the quasi-inverse. Let
 $(\ms{F},d^{1*}_0\ms{F}\cong d^{1*}_1\ms{F})$ where
 $\ms{F}\in\mr{Con}(Y)$ be a descent data.
 This is sent to
 \begin{equation*}
  \mr{Ker}\bigl(f^0_*\ms{F}\rightrightarrows(f\circ
   d^1_0)^0_*d^{1*}_0\ms{F}\bigr),
 \end{equation*}
 where $g_*^0$ denotes $\cH^0g_*$ for a morphism $g$.
 This functor is denoted by $\beta$.
 By adjunction, we have functors $\mr{id}\rightarrow\beta\circ\alpha$
 and $\alpha\circ\beta\rightarrow\mr{id}$, and it remains to check that
 these functors are equivalent. Since $f$ is assumed proper, by proper
 base change and \cite[1.3.7 (i)]{A}, we may assume that $X$ is a point.
 Further, by replacing $X$ by its finite extension, we may assume that
 $f$ has a section $s\colon X\rightarrow Y$. In this case, the argument
 is standard.
\end{proof}

\begin{cor}
 \label{desseco}
 Any proper hypercovering satisfies the descent.
\end{cor}
\begin{proof}
 The argument is very standard (for example, see
 \cite[Tag 0D8D]{Stack}), but we write a proof for the
 convenience of the reader. Let $Y_\bullet\rightarrow X$ be a
 hypercovering of $X$. If the hypercovering is
 $\mr{cosk}_1(Y_0\times_XY_0\rightrightarrows Y_0)$,
 then we already know the result by the
 lemma. The lemma also tells us that for a proper surjective morphism
 $W\rightarrow Z$, the pull-back $\mr{Con}(Z)\rightarrow\mr{Con}(W)$ is
 faithful. Thus, giving a descent data on $Y_\bullet$ is equivalent to
 giving a descent data on $\mr{cosk}_1(Y_1\rightrightarrows Y_0)$.
 From now on, we assume that
 $Y_\bullet=\mr{cosk}_1(Y_1\rightrightarrows Y_0)$.

 Given a proper hypercovering $Y_{\bullet}$, a descent data for
 $Y_\bullet$ is $\ms{F}\in\mr{Con}(Y_0)$ and an isomorphism
 $\phi\colon d^{1*}_0\ms{F}\cong d^{1*}_1\ms{F}$ satisfying some
 conditions.
 In order to define a descent data for
 $\mr{cosk}_1(Y_0\times_XY_0\rightrightarrows Y_0)$,
 we only need to descent $\phi$ to $Y_0\times_XY_0$. Now, we have the
 following morphism
 \begin{equation*}
  \alpha:=(\mr{pr}_1,\mr{pr}_2,s^0_0\circ d_1^1\circ\mr{pr_1})\colon
  Y_1\times_{(Y_0\times_XY_0)}Y_1\rightarrow
   Y_1\times Y_1\times Y_1
 \end{equation*}
 This defines the following diagram of simplicial schemes:
 \begin{equation*}
  \xymatrix{
   Y_1\times_{(Y_0\times_XY_0)}Y_1
   \ar@<0.7ex>[r]\ar[r]\ar@<-0.7ex>[r]_-{d'^2_i}
   \ar@{.>}[d]&
   Y_1\ar@{=}[d]
   \ar@<0.5ex>[r]\ar@<-0.5ex>[r]&
   Y_0\ar@{=}[d]\\
  Y_2
   \ar@<0.7ex>[r]^-{d^2_i}\ar[r]\ar@<-0.7ex>[r]&
   Y_1
   \ar@<0.5ex>[r]^-{d^1_i}\ar@<-0.5ex>[r]&
   Y_0
   }
 \end{equation*}
 where $d'^2_i:=\mr{pr}_i\circ\alpha$. By the universal property of
 $\mr{cosk}$, we have the dotted vertical arrow so that they form a
 morphism of simplicial schemes.
 The cocycle condition for $\phi$ on $Y_2$ pulled back to
 $Y_1\times_{(Y_0\times_XY_0)}Y_1$ by the dotted arrow gives us the
 following commutative diagram:
 \begin{equation*}
  \xymatrix@C=50pt{
   \mr{pr}_1^*d^{1*}_0\ms{F}
   \ar@{-}[r]^-{\mr{pr}_1^*\phi}_-{\sim}\ar@{-}[d]_{\sim}&
   \mr{pr}_1^*d^{1*}_1\ms{F}\ar@{-}[d]^{\sim}\\
  \mr{pr}_2^*d^{1*}_0\ms{F}
   \ar@{-}[r]^-{\mr{pr}_2^*\phi}_-{\sim}&
   \mr{pr}_2^*d^{1*}_1\ms{F}.
   }
 \end{equation*}
 Thus, the isomorphism $\phi$ descends to $Y_0\times_XY_0$, and defines
 a descent data on $\mr{cosk}_1(Y_0\times_XY_0\rightrightarrows Y_0)$.
 Finally, use Lemma \ref{descone} to conclude.
\end{proof}

\begin{lem}
 \label{vanlem}
 Let $X$ be a scheme, and $\ms{F}\in D^{\leq0}(X)$. We have
 $\H^0(\ms{F})=0$ is and only if for any closed immersion $i\colon
 Z\hookrightarrow X$, there exists a dense subscheme $U\subset Z$ such
 that $\H^0(i^*\ms{F})|_U=0$.
\end{lem}
\begin{proof}
 Only if part follows since $i^*$ is right exact by \cite[Proposition
 1.3.13]{AC}.
 Assume $\H^0\ms{F}$ is supported on a reduced scheme $Z$.
 Consider the triangle
 \begin{equation*}
  i_Z^*\tau_{<0}\ms{F}\rightarrow
   i_Z^*\ms{F}\rightarrow
   i_Z^*\H^0(\ms{F})\xrightarrow{+}.
 \end{equation*}
 Since $i^*$ is right exact,
 $\H^ii_Z^*\tau_{<0}\ms{F}=0$ for $i\geq0$, which implies that
 $\H^0i_Z^*(\ms{F})\cong\H^0i_Z^*\H^0(\ms{F})$. Since we assumed that
 $Z$ is the support of $\H^0\ms{F}$, we have
 $i_{Z*}\H^0i_Z^*\H^0(\ms{F})\cong\H^0(\ms{F})$. Combining these, we
 have $i_{Z*}\H^0i_Z^*(\ms{F})\cong\H^0(\ms{F})$, and this vanishes
 generically on $Z$ by assumption. This can happen only when
 $Z=\emptyset$.
\end{proof}

\subsection{}
\label{isocandsm}
Let us compare smooth objects with isocrystals.
For a scheme $X$ of finite type over $k$, we denote by
$\mr{Isoc}^\dag(X)$ the subcategory of the category of overconvergent
isocrystals on $X$ consisting of isocrystals whose constituents can be
endowed with Frobenius structure
(see right after \cite[1.1.3 (11)]{A}).
Caution that $\mr{Isoc}^\dag(X)$ is slightly smaller than the category
of overconvergent isocrystals on $X$.

\begin{thm*}
 Let $X$ be a scheme of finite type over $k$.
 Then we have a canonical equivalence of categories
 $\mr{Isoc}^\dag(X)\xrightarrow{\sim}\mr{Sm}(X)$.
 This equivalence is compatible with pull-back.
\end{thm*}
\begin{proof}
 First, let us construct the functor in the case where $X$ is smooth.
 In this situation, Caro \cite{Cafai} (cf.\ \cite[2.4.15]{A} for a
 summary) defines a fully faithful functor
 \begin{equation*}
  \rho_X\colon\mr{Isoc}^\dag(X)\rightarrow D(X).
 \end{equation*}
 Note that this functor is compatible with pull-back.
 All we need to show is that the essential image of this functor is
 $\mr{Sm}(X)$. First, let us check this claim when $X$ is a
 curve. This follows by \cite[Lemma 2.4.11]{A}.
 Now assume $X$ is smooth but not necessarily a curve.
 We show by using the induction on the dimension of $X$. We assume the
 equivalence is known for any smooth $X$ of dimension $<n$. Assume $X$
 is of dimension $n$. 
 Let $E\in\mr{Isoc}^\dag(X)$, and let us check that $\rho_X(E)$ is
 smooth. By definition and some limit argument, it suffices to check
 that for any morphism $c\colon C\rightarrow X$ from a smooth curve $C$,
 $c^*\rho_X(E)$ is smooth. However, since
 $c^*\rho_X(E)\cong\rho_X(c^*E)$ by the compatibility of pull-back and
 we have already checked the claim
 for curves, $\rho_X(c^*E)$ is smooth, thus $\rho_X(E)\in\mr{Sm}(X)$.
 Let $\ms{L}\in\mr{Sm}(X)$, and let us show that $\ms{L}$ comes from an
 isocrystal. There exists an open dense subscheme $j\colon U\subset X$
 such that $\ms{L}|_U\cong\rho_U(E_U)$. Let us show that $E_U$ extends
 to an isocrystal $E$ on $X$. In order to check this, it suffices to
 show that for any morphism $c\colon C\rightarrow X$ from a smooth curve
 $C$, $c^*E_U$ extends to an isocrystal on $C$ by Shiho's cut-by-curve
 theorem \cite{Scut}. By the compatibility of pull-back and the
 equivalence of $\rho_C$ we have already checked, $c^*E_U$ does extend
 to an isocrystal on $C$, and thus $E_U$ also extends to an isocrystal
 $E$ on $X$. To conclude the proof, we need to show that the
 isomorphism $\ms{L}|_U\cong\rho_U(E_U)$ extends uniquely to an
 isomorphism $\ms{L}\cong\rho_X(E)$.

 Let $i\colon Z\rightarrow X$ be the complement of $U$. Let us show that
 $\H^n\bigl(i^*\ms{L}\bigr)=0$. In order to show this, it suffices to
 check that for any closed immersion $i_W\colon W\hookrightarrow Z$,
 $\H^n\bigl(i^*_Wi^*\ms{L}\bigr)$ vanishes generically on $W$ by Lemma
 \ref{vanlem}. Since the associated reduced scheme of $W$ is generically
 smooth, $i^*_Wi^*\ms{L}$ is smooth on $W$, and by induction hypothesis,
 $i^*_Wi^*\ms{L}$ generically comes from an isocrystal.
 Since isocrystals concentrates on degree $\leq\dim(W)<n$, we get the
 claim. Now, since $\ms{L}$ is
 constructible, $\H^i\ms{L}=0$ for $i>n$.
 Considering the exact triangle
 $j_!\ms{L}\rightarrow\ms{L}\rightarrow i^*\ms{L}\xrightarrow{+}$, the
 homomorphism $\H^nj_!\ms{L}\rightarrow\H^n\ms{L}$ is surjective because
 we have checked that $\H^ni^*\ms{L}=0$.
 Thus, we have a canonical homomorphism $\ms{L}\rightarrow
 j_{!*}(\ms{L})$. This induces
 $\ms{L}\rightarrow j_{!*}\ms{L}\cong j_{!*}\rho_U(E_U)\cong
 \rho_X(E)$ whose restriction to $U$ is the given map.
 The compatibility of pull-back and induction hypothesis implies that
 this is in fact an isomorphism and is a unique homomorphism extending
 the given $\ms{L}|_U\cong\rho_U(E_U)$.

 In the general case, by using de Jong's alteration, we can take a
 proper hypercovering $Y_\bullet$ of $X$ such that $Y_i$ is smooth for
 any $i$. By Corollary \ref{desseco},
 proper descent of isocrystals \cite[Proposition 7.3]{S},
 and the compatibility of pull-back, we have a functor
 $\mr{Isoc}^\dag(X)\rightarrow\mr{Con}(X)$.
 It is easy to check that this functor does not depend on the choice of
 $Y_0$, $Y_1$ up to canonical isomorphism.
 The essential image
 coincides with $\mr{Sm}(X)$ since smooth objects are preserved by
 pull-back and Lemma \ref{smpullok}.
\end{proof}

\begin{cor}
 Let $X$ be a scheme, $j\colon U\hookrightarrow X$ an open immersion,
 and $\ms{L}\in\mr{Sm}(X)$. For any $\ms{G}\in
 D(X)$, we have a canonical isomorphism
 \begin{equation*}
  \ms{L}\otimes j_*(\ms{G})\xrightarrow{\sim}
   j_*\bigl(j^*\ms{L}\otimes\ms{G}\bigr).
 \end{equation*}
\end{cor}
\begin{proof}
 By limit argument, we may assume $X$ is of finite type over $k$.
 The homomorphism is defined by adjunctions. Since it is an isomorphism
 on $U$, it suffices to check that $i^!\bigl(\ms{L}\otimes
 j_*(\ms{G})\bigr)=0$ where $i\colon X\setminus U\rightarrow X$.
 Let us show the claim using the induction on the
 dimension of the support of $\ms{G}$. When the dimension is $0$, there
 is nothing to show.
 Take an alteration $g\colon X'\rightarrow X$ such that $X'$ is smooth,
 and let $j_{UV}\colon V\subset U$ be an open dense subscheme such
 that $g_V\colon g^{-1}(V)\rightarrow V$ is finite \'{e}tale. By
 induction hypothesis, we may assume that $\ms{G}=j_{UV*}\ms{G}_V$ for
 some $\ms{G}_V$. Since $g_V$ is finite \'{e}tale, $\ms{G}_V$ is a
 direct factor of $g_{V*}g_V^*\ms{G}_V$. Consider the following diagram
 \begin{equation*}
  \xymatrix{
   V'\ar[d]_{g_V}\ar@{}[rd]|\square\ar[r]^-{j'_V}&
   X'\ar[d]^{g}\ar@{}[rd]|\square&
   Z'\ar[d]^{g_Z}\ar[l]_-{i'}\\
  V\ar[r]^-{j_V}&X&X\setminus U.\ar[l]_-{i}
   }
 \end{equation*}
 Using projection formula and the commutation of $g_*$ and $i^!$, we
 have
 \begin{equation*}
  i^!\bigl(\ms{L}\otimes j_{V*}g_{V*}g_V^*(\ms{G}_V)\bigr)
   \cong
   i^!g_*\bigl(g^*\ms{L}\otimes j'_{V*}g_V^*(\ms{G}_V)\bigr)
   \cong
   g_{Z*}i'^!\bigl(g^*\ms{L}\otimes j'_{V*}g_V^*(\ms{G}_V)\bigr)
 \end{equation*}
 Thus, it suffices to check that $i'^!\bigl(g^*\ms{L}\otimes
 j'_{V*}g_V^*(\ms{G}_V)\bigr)=0$, and may assume that $X$ is
 smooth. Then we use Theorem \ref{isocandsm} and
 \cite[Proposition 5.8]{Aex} to conclude.
\end{proof}

\begin{cor}
 \label{commpsism}
 Let $(S,s,\eta)$ be a henselian trait.
 Let $\pi\colon X\rightarrow S$ be a morphism of finite type and
 $\ms{L}\in\mr{Sm}(X)$. Then for any $\ms{G}\in\mr{Hol}(X)$, we have a
 canonical isomorphism
 \begin{equation*}
  \Psi_\pi\bigl(\ms{L}\otimes\ms{G}\bigr)\cong
   \ms{L}|_{X_s}\otimes\Psi_\pi(\ms{G}).
 \end{equation*}
\end{cor}

\begin{thm}
 Let $f\colon X\rightarrow Y$ be a morphism of schemes.
 \begin{enumerate}
  \item\label{prespullsm}
       The functor $f^*$ preserves smooth objects, and induces a
       functor $f^*\colon\mr{Sm}(Y)\rightarrow\mr{Sm}(X)$.
  \item\label{prespullpr}
       Assume that $f$ is proper and smooth. Then $f_*$ preserves smooth
       objects.
 \end{enumerate}
\end{thm}
\begin{proof}
 The preservation under pull-back follows directly by definition, and we
 write it just for the future reference. Let us check the second claim.
 Take a strict henselian trait $S$ and a morphism $\phi\colon
 S\rightarrow Y$. Using \ref{prespullsm} and the commutation of $f_*$
 and $\phi^*$ because $f$ is assumed proper, it suffices to check the
 claim when $Y$ is strict henselian trait.
 Let $\ms{L}$ be a smooth object on $X$. We have
 \begin{equation*}
  \Phi_{\mr{id}} f_*(\ms{L})\cong
   f_*\Phi_f(\ms{L})\cong
   f_*\bigl(\Phi_f(L_X)\otimes\ms{L}|_{X_s}\bigr),
 \end{equation*}
 where the first isomorphism follows by Corollary
 \ref{fundcomm}.\ref{fundcommpr}, and the second by
 Corollary \ref{commpsism}. Finally, since $f$ is assumed smooth, we
 have $\Phi_f(L_X)\cong f^*\Phi_{\mr{id}}(L_Y)$ by
 Corollary \ref{fundcomm}.\ref{fundcommsm}. Since $Y$
 is assumed strict henselian trait,
 $\mr{rk}\bigl(\Psi_{\mr{id}}(L_Y)\bigr)=1$ by Proposition
 \ref{dim0case}. Thus, the exact triangle \ref{fundcomm}.\ref{fundextr}
 tells us that $\Phi_{\mr{id}}(L_Y)=0$, and the theorem follows.
\end{proof}

\begin{rem*}
 One can think part \ref{prespullpr} of the theorem above as a
 $\ms{D}$-module theoretic version of Berthelot's conjecture
 \cite[(4.3)]{B}.
 This variant has already been considered by Caro in
 \cite[Th\'{e}or\`{e}me 4.4.2]{Cprev} when $X$, $Y$ are realizable
 schemes. In that case, he proved without the
 existence of Frobenius structure, whereas we assume the existence
 implicitly in the construction of the category $\mr{Hol}(X)$.
 However, our theorem is stronger in the sense that the schemes need not
 be realizable. In order to deduce the original Berthelot's conjecture
 from our result, one might need to compare our push-forward and
 relative rigid cohomology (cf.\ \cite[2.4.16]{A}).
 This will be addressed in future works.
\end{rem*}

\begin{prob}
 Assume $X$ is smooth. For any object $\ms{M}\in D(X)$, the
 characteristic variety $\mr{Car}(\ms{M})$ is defined as a closed
 subscheme of codimension $\dim(X)$ in $T^*X$ by Berthelot.
 We expect that this characteristic variety has the following
 characterization: for any $X\supset U\xrightarrow{f}\mb{A}^1$ such that
 $U$ is an open subscheme of $X$ and
 $df(U)\cap\mr{Car}(\ms{M})=\emptyset$ we have $\Phi_f(\ms{M})=0$, and
 $\mr{Car}(\ms{M})$ is the smallest closed subscheme of $T^*X$ having
 possessing such a property.
\end{prob}

\section{Toward a local theory}
\label{locsec}
\begin{dfn}
 \label{localcat}
 \begin{enumerate}
  \item\label{localcatbase}
       Let $S$ be a henselian local scheme ({\it i.e.}, the spectrum of
       a henselian local ring), and let $i\colon s\hookrightarrow S$ be
       the closed immersion from the closed point.
       The category $\mr{Loc}(S)$ is defined as follows:
       The object is the same as $\mr{Hol}(S)$. For
       $\ms{F}\in\mr{Hol}(S)$ the corresponding object in $\mr{Loc}(S)$
       is denoted by $\ms{F}^{\mr{loc}}$. Then
       \begin{equation*}
	\mr{Hom}_{\mr{Loc}(S)}(\ms{F}^{\mr{loc}},\ms{G}^{\mr{loc}}):=
	 H^0\bigl(s,i^*\shom_S(\ms{F},\ms{G})\bigr),
       \end{equation*}
       where $H^i(s,\ms{H}):=\mr{Hom}_{D(s)}(L_s,\ms{H}[i])$.
       The object $L_S^{\mr{loc}}$ is denoted by $L_S$ for simplicity.

  \item Let $S$ be a local henselian scheme, and $f\colon X\rightarrow
	S$ be a morphism of finite type. Then $\mr{Loc}(X/S)$ is defined
	as follows. The objects are the same as $\mr{Hol}(X)$. For
	$\ms{F},\ms{G}\in\mr{Hol}(X)$, we define
	\begin{equation*}
	 \mr{Hom}_{\mr{Loc}(X/S)}(\ms{F},\ms{G}):=
	  \mr{Hom}_{\mr{Loc}(S)}
	  \bigl(L_S,f_*\shom(\ms{F},\ms{G})\bigr).
	\end{equation*}
 \end{enumerate}
\end{dfn}

\begin{rem}
 \begin{enumerate}
  \item The category $\mr{Loc}(X/S)$ is certainly an additive category.
	However, we do not know if this is abelian or not.
       
  \item\label{highhomzero}
       We assume we are in the situation of
	\ref{localcat}.\ref{localcatbase}.
	We remark that
	\begin{equation*}
	 H^i\bigl(s,i^*\shom_S(\ms{F},\ms{G})\bigr)=0
	\end{equation*}
	for $i<0$. Indeed, let us check first that for any scheme $X$
	and $\ms{F},\ms{G}\in\mr{Hol}(X)$,
	$\cH^i\shom_X(\ms{F},\ms{G})=0$ for $i<0$.
	We may assume $X$ is smooth. For $\ms{F}\in
	D^{\leq0}(X)$, $\ms{G}\in D^{\geq0}(X)$, we may check as
	\cite[2.1.20]{BBD} that $\shom(\ms{F},\ms{G})\in
	D^{\geq\dim(X)}(X)$. Thus,
	$\cH^i\shom(\ms{F},\ms{G})=0$ for $i<0$. Since $i^*$ is
	c-t-exact and $L_s$ is constructible object, we get the claim.
	
  \item Ultimately, we expect a triangulated category
	$D_{\mr{loc}}(X/S)$ with the following properties:
	First, we have a functor $\rho\colon D(X)\rightarrow
	D_{\mr{loc}}(X/S)$. For $\ms{F},\ms{G}\in D(X)$, we should have
	\begin{equation*}
	 \mr{Hom}_{D_{\mr{loc}}(X/S)}
	  \bigl(\rho(\ms{F}),\rho(\ms{G})\bigr)
	  \cong
	  \mr{Hom}_{D(s)}
	  \bigl(L_s,i^*f_*\shom(\ms{F},\ms{G})\bigr).
	\end{equation*}
	Secondly, we have a t-structure on $D_{\mr{loc}}(X/S)$
	whose heart contains $\rho(\mr{Hol}(X))$. The computation of
	part \ref{highhomzero} above shows that the ``higher
	homotopies'' of $\rho(\mr{Hol}(X))$ vanish in the category
	$D_{\mr{loc}}(X/S)$.
	This gives us an evidence of the existence of such t-structure.
	The category $\mr{Loc}(X/S)$ should be a full subcategory of
	this heart.
	
	This category should be an analogue of the derived category of
	constructible sheaves for $\ell$-adic sheaves of a scheme of
	separated of finite type over a local henselian scheme as in
	\cite[Theorem 6.3]{eke}. The following Theorem \ref{loccarob}
	gives an evidence for this philosophy.
 \end{enumerate}
\end{rem}

\subsection{}
Let $S$ be a strict henselian trait, and $\pi\colon X\rightarrow S$ be a
morphism of finite type. The nearby cycle formalism extends to that on
$\mr{Loc}(X)$. For $\ms{F},\ms{G}\in D(X)$, we have the canonical
homomorphism
\begin{equation*}
 \Psi_\pi(\ms{F})\otimes\ms{G}|_{X_s}\rightarrow
  \Psi_\pi\bigl(\ms{F}\otimes\ms{G}\bigr).
\end{equation*}
On the other hand, the adjunction induces a map
\begin{equation*}
 L_{X_s}\boxtimes H^0(X_s,\ms{H})\rightarrow\ms{H}
\end{equation*}
for any $\ms{H}\in D(X_s)$. Combining these, we obtain a homomorphism
\begin{equation*}
 \Psi_\pi(\ms{F})\boxtimes
  \mr{Hom}_{\mr{Loc}(X)}(\ms{F}^{\mr{loc}},\ms{G}^{\mr{loc}})
  \rightarrow
  \Psi_\pi\bigl(\ms{F}\otimes\shom(\ms{F},\ms{G})\bigr)
  \rightarrow
  \Psi_\pi(\ms{G}).
\end{equation*}
Thus, we have a homomorphism
\begin{equation*}
 \mr{Hom}_{\mr{Loc}(X)}(\ms{F}^{\mr{loc}},\ms{G}^{\mr{loc}})
  \rightarrow
  \mr{Hom}_{\mr{Hol}(X_s)}\bigl(\Psi_\pi(\ms{F}),
  \Psi_\pi(\ms{G})\bigr).
\end{equation*}
If $\ms{F}=\ms{G}$, then the identity is sent to the identity, and the
map is compatible with the composition. Thus, $\Psi_\pi$ is defined also
on the level of the category $\mr{Loc}(X)$.

\subsection{}
Let us describe the category $\mr{Loc}(S)$ when $S$ is a henselian
trait in terms of the theory on a formal unit disk after Crew.
Let $\ms{S}:=k\dd{t}$, a formal disk. For this disk, he constructed the
category $\mr{Coh}^{\mr{an}}(\mc{D}^\dag)$ in \cite[5.2]{Cr}.
Crew defined a category of holonomic objects in
$\mr{Coh}^{\mr{an}}(\mc{D}^\dag)$ with Frobenius structure, and denote
it by $\mr{Hol}^{\mr{an}}(F\mc{D}^\dag)$. As usual, we consider the full
subcategory of $\mr{Coh}^{\mr{an}}(\mc{D}^\dag)$ consisting of objects
which are of finite length and whose constituent can be endowed with a 
Frobenius structure with which the constituent is in
$\mr{Hol}^{\mr{an}}(F\mc{D}^\dag)$. We denote this category by
$\mr{Hol}^{\mr{an}}(\mc{D}^\dag)$, or
$\mr{Hol}^{\mr{an}}(\mc{D}^\dag_{\ms{S},\mb{Q}})$ if we want to
emphasize the formal disk. Even though the definition seems to be a bit
involved, $\mr{Hol}^{\mr{an}}(\mc{D}^\dag)$ is very close to the
category which has already appeared in the theory of $p$-adic
differential equation. In fact, the category
$\mr{MLS}(\mc{R},\mathbf{F},\mathbf{pot})$ appearing in
\cite[D\'{e}finition 6.0-19]{CM} is contained full faithfully in
$\mr{Hol}^{\mr{an}}(\mc{D}^\dag)$, and it is easy to characterize this
subcategory: it consists of the objects
$M\in\mr{Hol}^{\mr{an}}(\mc{D}^\dag)$ such that $i^!M=0$.
The verification of this characterization is left to the reader.

\begin{thm*}
 \label{loccarob}
 Let $S$ be a henselian trait such that the closed point is of finite
 type over $k$. Let $\ms{S}$ be the formal completion of $S$ with
 respect to the closed point. Then we have the canonical equivalence of
 categories
 $\mr{Loc}(S)\cong\mr{Hol}^{\mr{an}}(\mc{D}^\dag_{\ms{S},\mb{Q}})$.
\end{thm*}
\begin{proof}
 First, we have the canonical functor
 $\mr{An}\colon\mr{Hol}(S)\rightarrow
 \mr{Hol}^{\mr{an}}(\mc{D}^\dag)$ sending
 the object $\ms{F}\in\mr{Hol}(S)$ to $\ms{F}^{\mr{an}}$.
 Let us show that this functor factors through
 $\mr{Hol}(S)\rightarrow\mr{Loc}(S)$. In order to check this, let us
 show that we have a canonical isomorphism
 $H^i(S,\mc{M})\xrightarrow{\sim}H^i(s,i^+\mc{M})$ for
 $\mc{M}\in\mr{Hol}^{\mr{an}}(\mc{D}^\dag)$ and $i\in\mb{Z}$.
 Since the homomorphism can be defined by adjointness, we need to check
 that it is an isomorphism. When $\mc{M}=i_+\mc{N}$, the claim is
 obvious, so we may assume that $\mc{M}=j_+j^+\mc{M}$.
 In this case, the claim follows by \cite[Lemma 3.1.10]{AM}.
 Now, this isomorphism induces
 \begin{align*}
  \mr{Hom}(\ms{F}^{\mr{loc}},\ms{G}^{\mr{loc}}):=
  H^0\bigl(s,i^*\shom(\ms{F},\ms{G})\bigr)
  \xleftarrow{\sim}
  H^0\bigl(S,\shom(\ms{F}^{\mr{an}},\ms{G}^{\mr{an}})\bigr)
  \cong
  \mr{Hom}(\ms{F}^{\mr{an}},\ms{G}^{\mr{an}}).
 \end{align*}
 It is easy to check that this isomorphism is compatible with
 composition, and we have the desired functor
 $\mr{Loc}(S)\rightarrow\mr{Hol}^{\mr{an}}(\mc{D}^\dag)$, which is
 moreover fully faithful. To check the equivalence, it remains to check
 that the functor is essentially surjective. This follows by
 Crew-Matsuda extension \cite[Theorem 8.2.1]{Cr}.
\end{proof}

\begin{rem}
 \begin{enumerate}
  \item It might be possible to reprove \cite{AM} without using
	microlocal technique by using the foundation of this paper.
  \item  It would be interesting to compare the theory developed here
	 and recent works of Lazda-Pal \cite{LP}, Caro-Vauclair
	 \cite{CV}, or Crew \cite{Cr2} on the theory of $p$-adic
	 cohomology theory for formal schemes or schemes over Laurent
	 series field.
 \end{enumerate}
\end{rem}

Tomoyuki Abe:\\
Kavli Institute for the Physics and Mathematics of the Universe
(WPI), University of Tokyo\\
5-1-5 Kashiwanoha, Kashiwa, Chiba, 277-8583, Japan\\
{\tt tomoyuki.abe@ipmu.jp}

\end{document}